\documentclass[a4paper,11pt]{article}
\usepackage{amsmath,amssymb,amsthm}
\title{\bf Bloch's Theorem in the Context of Quaternion Analysis\footnote{The present article is a preliminary version, submitted to Computational Methods and Function Theory.}}
\author{K. G\"urlebeck\thanks{Bauhaus-Universit\"at Weimar, Institut f\"ur Mathematik/Physik, Coudraystr. 13B, D-99421 Weimar, Germany. Email: klaus.guerlebeck@uni-weimar.de} \, and 
J. Morais\thanks{Freiberg University of Mining and Technology, Institute of Applied Analysis, D-09596 Freiberg, Germany. Email: joao.pedro.morais@ua.pt}}
\date{}
\newtheorem{Theorem}{Theorem}[section]
\newtheorem{Lemma}{Lemma}[section]
\newtheorem{Definition}{Definition}[section]
\newtheorem{Remark}{Remark}[section]
\newtheorem{Proposition}{Proposition}[section]

\usepackage{amsfonts}
\begin{document}
\maketitle
\begin{abstract}
The classical theorem of Bloch (1924) asserts that if $f$ is a holomorphic function on a region that contains the closed unit disk $|z|\leq 1$ such that $f(0) = 0$ and $|f'(0)| = 1$, 
then the image domain contains discs of radius $\frac{3}{2}-\sqrt{2} > \frac{1}{12}$. The optimal value is known as Bloch's constant and  $\frac{1}{12}$ is not the best possible. In this paper we give a direct generalization of Bloch's theorem to the three-dimensional Euclidean space in the framework of quaternion analysis. We compute explicitly a lower bound for the Bloch constant.
\end{abstract}


\noindent {\bf MSC Subject-Classification}: 30G35, 32A05.

\section{Introduction and statement of results}

Quaternion analysis is a higher dimensional function theory offering both a generalization of complex analysis in the plane and a refinement of classical harmonic analysis. The rich structure of this function theory involves the study of quaternion-valued functions that are defined in open subsets of $\mathbb{R}^{n}$ ($n=3, 4$) and that are solutions of generalized Cauchy-Riemann or Dirac systems. They are often called monogenic functions. Yet quaternion analysis has become a major research area in mathematics having connections with boundary value problems and partial differential equations theory or other fields of physics and engineering. For a thorough treatment of this function theory, the reader is referred to \cite{GS1989,GS1997,KS1996,K2003,ShapiroVasilevski11995,ShapiroVasilevski21995}.

In complex analysis much effort has been placed in the study of the classical Schwarz's lemma during the last century, starting for example with the famous works of Schwarz \cite{Schwarz1890}, Pick \cite{Pick1916}, Ahlfors \cite{Ahlfors1938} and Carath\'eodory \cite{Caratheodory1952}, and many others. Their analyses are not only useful auxiliary tools, but even provide powerful information to study classical problems of the theory of conformal maps, so that this research domain has developed into a field of central and vast interest within complex function theory. In addition to various general problems in the geometric theory of holomorphic functions, it naturally embraces Cauchy's inequalities, maximum modulus principle, versions of Schwarz-Pick and Bohr theorems, as well as Bloch's theorem, which are fundamental results with important consequences \cite{MindaSchober1983}. Theorems of this type seem to be more and more involved in progress in conformal geometry in higher dimensions, and in particular in the quaternion analysis setting. As a first step towards in a series of papers \cite{GueJoao2007,GueJoao2009,GueJoao2011} the authors investigated an higher dimensional counterpart of Bohr's phenomenon in the context of quaternion analysis. A lot of deeper results and extended list of references concerning this theorem for monogenic functions in $\mathbb{R}^3$, as well as its different modifications, can be found in \cite{JoaoThesis2009} Ch.3 (cf. also \cite{GueJoao2011}).

Recently, particular attention has also been devoted to the generalization of Bloch's theorem to higher dimensions. It has been conjectured by Eremenko \cite{Eremenko2000} that there exists a version of Bloch's theorem for $K$-quasiregular mappings on the unit ball, as well as for quasimeromorphic mappings. We mention that Bloch's theorem in connection with quasiregular holomorphic mappings in several complex variables has also been studied in \cite{ChenGauthier2001}. In this article we shall pose the question of whether Bloch theorem can be generalized to the context of quaternion analysis. We confine ourselves to an examination of the image domain of a monogenic function defined in a ball of $\mathbb{R}^3$ with values in the reduced quaternions (identified with $\mathbb{R}^3$). This class of functions coincides with the solutions of the well known Riesz system and shows more analogies to complex holomorphic functions than the more general class of quaternion-valued monogenic functions. We shall say Bloch's original proof \cite{Bloch1925} depended on the theory of the comparison of two power series used by Wiman \cite{Wiman1914} in the case of integral functions. Almost simultaneously, and working independently Landau and Valiron simplified Bloch's arguments considerably \cite{LandauValiron1929}. There are many other proofs of Bloch's theorem, including works by Landau \cite{Landau1926}, Carath\'eodory \cite{Caratheodory1929}, Heins \cite{Heins1962} and Pommerenke \cite{Pommerenke1964}. The reference list does not claim to be complete. Further references can be found in the books \cite{Ahlfors1973,Remmert1998}. Here we follow closely the proof given by Estermann \cite{Estermann1971} because of its geometric character. With little fundamental alteration his proof is considerably simplified compared to the previous ones, so that he establishes some estimates for the Fourier coefficients of a holomorphic function by the growth of the maximum modulus of its complex derivative.

One main reason for pursuing this direction is that a sufficiently well developed theory already exists for the construction of an appropriate monogenic Fourier series by means of quaternion analysis tools. In \cite{Cacao2004,IM2006} a few structural properties of the complex Fourier series expansion could be generalized in this context (see also \cite{GueJoao2009,JoaoThesis2009,JoaoGue2011} for more details). Therein, one could obtain, similar to the complex case, explicit series representations of the hypercomplex derivative and primitive based on a monogenic Fourier series expansion in terms of solid spherical monogenics. This gives us an interesting way of motivating addition itself. A second reason is that while the underlying theorem is essentially differential geometric in character, we manage to give a purely function theoretic proof for the case of monogenic functions in $\mathbb{R}^3$. There are a few attempts to generalize Bloch's theorem to higher dimensions. Without claiming completeness we mention here the paper by Rochon \cite{Rochon2001}, who stated a Bloch-type theorem for hyperholomorphic functions with values in the bicomplex numbers. Earlier, already Wu proved in \cite{Wu1967} a Bloch-type theorem for quasiconformal mappings in $\mathbb{C}^n$. Both approaches as well as the majority of proofs of Bloch's theorem make use of the commutativity of the multiplication in the underlying field of coefficients or in the algebra of holomorphic functions. With respect to the already mentioned geometric background it is also of theoretical interest to see whether a Bloch-type theorem can be proved if the underlying structure is not commutative as in the case of quaternions and quaternion-valued monogenic functions. This understanding can be the basis for more generalizations of Bloch's theorem. We will not consider concrete applications of Bloch's theorem in this paper.

\section{Basic notions}

This section is devoted to the exposition of some basic algebraic facts about real quaternions, which we use throughout of this paper. For all what follows we will work in $\mathbb{H}$, the skew field of real quaternions. This means we can express each element $\mathbf{z} \in \mathbb{H}$ uniquely in the form $\mathbf{z}=z_0 + z_1 \textbf{i} + z_2 \textbf{j} + z_3 \textbf{k}$, with real numbers $z_i$ $(i=0,1,2,3)$, where the imaginary units $\textbf{i}$, $\textbf{j}$, and $\textbf{k}$ stand for the elements of the basis of $\mathbb{H}$, subject to the multiplication rules
\begin{eqnarray*}
\textbf{i}^2 = \textbf{j}^2 = \textbf{k}^2 = -1; \quad \; \textbf{i} \textbf{j} = \textbf{k} = - \textbf{j} \textbf{i}, \quad \;
\textbf{j} \textbf{k} = \textbf{i} = - \textbf{k} \textbf{j}, \quad \;
\textbf{k} \textbf{i} = \textbf{j} = - \textbf{i} \textbf{k}.
\end{eqnarray*}

As usual, the real vector space $\mathbb{R}^4$ may be embedded in $\mathbb{H}$ by identifying the element $z := (z_0,z_1,z_2,z_3) \in \mathbb{R}^4$ with $\mathbf{z} := z_0 + z_1 \mathbf{i} + z_2 \mathbf{j} + z_3 \mathbf{k} \in \mathbb{H}$. In the sequel, consider the subset $\mathcal{A} := {\rm span}_{\mathbb{R}}\{1,\mathbf{i},\mathbf{j}\}$ of $\mathbb{H}$. Then, the real vector space $\mathbb{R}^3$ may be embedded in $\mathcal{A}$ via the identification of $x := (x_0,x_1,x_2)=(x_0,\underline{x}) \in \mathbb{R}^3$ with the reduced quaternion $\textbf{x} := x_0 + x_1 \textbf{i} + x_2 \textbf{j} \in \mathcal{A}$. As a matter of fact, throughout the text we will often use the symbol $x$ to represent a point in $\mathbb{R}^3$ and $\mathbf{x}$ to represent the corresponding reduced quaternion. Also, we emphasize that $\mathcal{A}$ is a real vectorial subspace, but not a subalgebra, of $\mathbb{H}$. For any $\textbf{x} := x_0 + x_1 \textbf{i} + x_2 \textbf{j} \in \mathcal{A}$ we write $\overline{\mathbf{x}}$ for $x_0 - x_1 \textbf{i} - x_2 \textbf{j}$, and call it the quaternion conjugate of $\textbf{x}$. Also $|\mathbf{x}|$ is $\sqrt{\mathbf{x} \overline{\mathbf{x}}} = \sqrt{\overline{\mathbf{x}}\mathbf{x}}$, the non-negative square root of $\mathbf{x} \overline{\mathbf{x}} = \overline{\mathbf{x}}\mathbf{x} = x_0^2+x_1^2+x_2^2$. This number is called the norm of $\mathbf{x}$, and it coincides with the corresponding Euclidean norm of $x$ as a vector in $\mathbb{R}^3$. Yet $x_0$ is called the scalar part of $\mathbf{x}$, $x_1 \mathbf{i} + x_2 \mathbf{j}$ the vector part of $\mathbf{x}$, and we write $x_0 = \mathbf{Sc}(\mathbf{x})$, $x_1 \mathbf{i} + x_2 \mathbf{j} = \mathbf{Vec}(\mathbf{x})$. We shall always assume the quaternion $0+0\textbf{i}+0\textbf{j}:=\mathbf{0}_{\mathcal{A}}$ to be the neutral element of addition in the sequel.

Now, let $\Omega$ be an open subset of $\mathbb{R}^3$ with a piecewise smooth boundary. The standard form of a reduced quaternion-valued function or, briefly, an $\mathcal{A}$-valued function, will be taken to be
\begin{eqnarray*}
\mathbf{f} : \Omega \longrightarrow \mathcal{A}, \;\;\;\; \mathbf{f}(x) = [\mathbf{f}(x)]_0 + [\mathbf{f}(x)]_1 \textbf{i} + [\mathbf{f}(x)]_2 \textbf{j},
\end{eqnarray*}
where $[\mathbf{f}]_i$ $(i=0,1,2)$ are real-valued functions defined in $\Omega$. Properties such as continuity, differentiability, integrability, and so on, which are ascribed to $\mathbf{f}$ have to be fulfilled by all components $[\mathbf{f}]_i$. Let $B_r(0):=B_r$ be the ball of radius $r$ in $\mathbb{R}^3$ centered at the origin. We further introduce the real-linear Hilbert space of square integrable $\mathcal{A}$-valued functions defined on $B_r$, that we denote by $L_2(B_r;\mathcal{A};\mathbb{R})$. In this assignment, the scalar inner product is defined by
\begin{eqnarray} \label{InnerProduct}
<\mathbf{f},\mathbf{g}>_{L_2(B_r;\mathcal{A};\mathbb{R})} \, = \int_{B_r}\;{\mathbf{Sc}}({\overline{\mathbf{f}}\,\mathbf{g}) \, dV_r} \,,
\end{eqnarray}
where $dV_r$ denotes the Lebesgue measure on $B_r$.

Matters become interesting when we consider the notion of monogenicity, which is introduced by means of the so-called generalized Cauchy-Riemann operator
\begin{eqnarray} \label{CauchyRiemannOperator}
D = \partial_{x_0} + \textbf{i} \,\partial_{x_1} + \textbf{j} \,\partial_{x_2}.
\end{eqnarray}

\begin{Definition} \rm {(Monogenicity)}
A continuously real-differentiable $\mathcal{A}$-valued function $\mathbf{f}$ is called monogenic in $\Omega$ if $D\mathbf{f}=\mathbf{0}_{\mathcal{A}}$ in $\Omega$.
\end{Definition}

As the generalized Cauchy-Riemann operator $(\ref{CauchyRiemannOperator})$ and its conjugate
\begin{eqnarray*}
\overline{D} = \partial_{x_0} - \textbf{i} \,\partial_{x_1} - \textbf{j} \,\partial_{x_2}
\end{eqnarray*}
factorize the Laplace operator in $\mathbb{R}^3$ in the sense that $\Delta_3 = D \overline{D} = \overline{D} D$, it follows that a monogenic function in $\Omega$ is harmonic in $\Omega$, and so are all its components.

An additional advantage is the realization that any monogenic $\mathcal{A}$-valued function is two-sided monogenic. In other words, this means it satisfies simultaneously the 
equations $D\mathbf{f}=\mathbf{f}D=\mathbf{0}_{\mathcal{A}}$, which are equivalent to the system
\begin{eqnarray*}
{\rm (R)} \left\{ \begin{array} {ccc}
{\rm div} \overline{\mathbf{f}} &=& 0 \\[1.0ex]
{\rm curl} \overline{\mathbf{f}} &=& 0
\end{array}\right. 
\, \Longleftrightarrow \,
\left\{ \begin{array} {ccc}
\displaystyle \partial_{x_0} [\mathbf{f}]_0 - \sum_{i=1}^2 \partial_{x_i} [\mathbf{f}]_i = 0 && \\
\hspace{0.63cm} \partial_{x_j} [\mathbf{f}]_i + \partial_{x_i} [\mathbf{f}]_j = 0 && (i \neq j, \, 0 \leq i, j \leq 2).
\end{array}\right.
\end{eqnarray*}

The system (R) is known as Riesz system \cite{Riesz1958}. It clearly generalizes the classical Cauchy-Riemann system for holomorphic functions in the complex plane. Following \cite{Leutwiler2001}, the solutions of the system (R) are called (R)-solutions. The subspace of polynomial (R)-solutions of degree $n$ will be denoted by $\mathcal{R}^+(B_r;\mathcal{A};n)$. In \cite{Leutwiler2001}, it is shown that the space $\mathcal{R}^+(B_r;\mathcal{A};n)$ has dimension $2n+3$. We also denote by $\mathcal{R}^+(B_r;\mathcal{A}):=L_2(B_r;\mathcal{A}; \mathbb{R}) \cap \ker D$ the space of square integrable $\mathcal{A}$-valued monogenic functions defined in $B_r$.\\

Ultimately, we recall some fundamental definitions and notations which will be needed through the text.
\begin{Definition} {\rm (Hypercomplex Derivative, see \cite{GueMal1999,MitelmanShapiro1995,Sud1979})}
Let $\mathbf{f}$ be a continuously real-differentiable $\mathcal{A}$-valued function, $(\frac12 \overline{D}) \mathbf{f}$ is called hypercomplex derivative of $\mathbf{f}$.
\end{Definition}

\begin{Definition} {\rm (Hyperholomorphic constant)}
An $\mathcal{A}$-valued monogenic function with an identically vanishing hypercomplex derivative is called hyperholomorphic constant.
\end{Definition}

\begin{Definition} {\rm (Hypercomplex Primitive)}
A continuously real-differentiable $\mathcal{A}$-valued function $\mathbf{F}$ is called monogenic primitive of an $\mathcal{A}$-valued monogenic function $\mathbf{f}$ with respect to the hypercomplex derivative, if $\mathbf{F} \in \ker D$ and $(\frac12 \overline{D}) \mathbf{F}=\mathbf{f}$. For a given $\mathbf{f} \in \ker D$, if such function $\mathbf{F}$ exists, we denote 
$\mathbf{F}:=\mathcal{P}(\mathbf{f})$.
\end{Definition}

Related to the previous paragraph, we will use an operator approach as for instance illustrated in \cite{Cacao2004}, where a primitivation operator acting on each element of an orthogonal basis is defined and extended by continuity to the whole space. Analogously to the complex case we speak about the primitive of some given function $\mathbf{f}$ if it is the result of the application of the operator $\mathcal{P}$ to $\mathbf{f}$, without adding any hyperholomorphic constant. "Omitting" the constants means in fact to look for the (unique) primitive that is orthogonal to the hyperholomorphic constants (see e.g. \cite{GueJoao2011}).

\section{A special system of homogeneous monogenic polynomials as solutions of the Riesz system in $\mathbb{R}^3$}

The following constructions are based on the introduction of a standard system of spherical harmonics as shown e.g. in \cite{Sansone1959}. We use spherical coordinates,
\begin{equation*}
x_0 = r \cos \theta, \; x_1 = r \sin \theta \cos \varphi, \; x_2 = r
\sin \theta \sin \varphi,
\end{equation*} where
$0 < r < \infty$, $0 < \theta \leq \pi$, and $0 < \varphi \leq 2\pi$. We consider the set of homogeneous harmonic polynomials,
\begin{equation} \label{HHP}
\{ r^{n+1} U^0_{n+1}, r^{n+1} U^m_{n+1}, r^{n+1} V^m_{n+1},
m=1,...,n+1  \}_{n \in \mathbb{N}_0}
\end{equation}
formed by the extensions in the ball of the spherical harmonics 
\begin{eqnarray} \label{sphericalharmonics}
U^0_{n+1}(\theta,\varphi) &=& P_{n+1}(\cos \theta) \nonumber \\[0.5ex]
U^m_{n+1}(\theta,\varphi) &=& P^m_{n+1}(\cos \theta) \cos (m \varphi) \\[0.5ex]
V^m_{n+1}(\theta,\varphi) &=& P^m_{n+1}(\cos \theta) \sin (m \varphi), \qquad m=1, \dots ,n+1 \nonumber.
\end{eqnarray}
Here, $P_{n+1}$ stands for the Legendre polynomial of degree $n+1$ and the functions $P^m_{n+1}$, where $m=1, \ldots ,n+1,$ are the associated Legendre functions.
In \cite{Cacao2004} and \cite{IGS2006}, a special $\mathbb{R}$-linear complete orthonormal system of $\mathcal{A}$-valued homogeneous monogenic polynomials in the unit ball of $\mathbb{R}^3$ is explicitly constructed by applying the operator $\frac12\overline{D}$ to the system (\ref{HHP}).

Restricting the resulting solid spherical monogenics to the surface of the unit ball we obtain a system of spherical monogenics, denoted by
\begin{eqnarray*}
\left\{ \mathbf{X}^{0}_n, \mathbf{X}^{m}_n, \mathbf{Y}^{m}_n : m=1,\ldots,n+1\right\}_{n \in \mathbb{N}_0}
\end{eqnarray*}

This system can be seen as a refinement of the conventional spherical harmonics, and correspondingly it constitutes an extension of the role of the well known Chebyshev and Legendre polynomials (resp. associated Legendre functions) as shown in \cite{JoaoGue2011}.  More importantly, it can be explicitly constructed by using recurrence relations and preserves some basic properties in common with holomorphic $z$-powers. The fundamental references for the preceding arguments and explicit expressions of these special spherical monogenics are \cite{GueJoao22011,JoaoThesis2009,JoaoGue2011}. For the purposes of this paper, we will give a different, less elaborate exposition of their features.

We recall from \cite{Cacao2004} and \cite{CGBHanoi2005} the following properties:
\begin{enumerate}
\item[1.] The functions $\mathbf{X}^{0,\dagger}_n:=r^n \mathbf{X}^{0}_n$, $\mathbf{X}^{m,\dagger}_n:=r^n \mathbf{X}^{m}_n$, and $\mathbf{Y}^{m,\dagger}_n:=r^n \mathbf{Y}^{m}_n$ are homogeneous monogenic polynomials;\\[-1.5ex]
\item[2.] For each $n=0,1,\ldots$, the polynomials $\mathbf{X}^{0,\dagger}_n$, $\mathbf{X}^{m,\dagger}_n$, $\mathbf{Y}^{m,\dagger}_n$ $(m=1,\ldots,n+1)$ form a complete orthogonal system in $\mathcal{R}^{+}(B_r;\mathcal{A})$, and their norms are explicitly given by
\begin{eqnarray*}
&& \|\mathbf{X}^{0,\dagger}_n\|_{L_2(B_r;\mathcal{A};\mathbb{R})} \,= \, \sqrt{\frac{r^{2n+3}}{2n+3}} \sqrt{\pi \, (n+1)}, \\[0.5ex]
&& \|\mathbf{X}^{m,\dagger}_n\|_{L_2(B_r;\mathcal{A};\mathbb{R})} \,= \, \|\mathbf{Y}^{m,\dagger}_n\|_{L_2(B_r;\mathcal{A};\mathbb{R})} \\
&& \hspace{2.95cm} = \, \sqrt{\frac{r^{2n+3}}{2n+3}} \sqrt{\frac{\pi}{2} (n+1) \frac{(n+1+m)!}{(n+1-m)!}};\\[-1.5ex]
\end{eqnarray*}
\item[3.] For $n \geq 1$, we have $(\frac{1}{2} \overline{D}) \mathbf{X}_n^{l,\dagger} = (n+l+1) \mathbf{X}_{n-1}^{l,\dagger}$ $(l=0,\ldots,n)$ and $(\frac{1}{2} \overline{D}) \mathbf{Y}_n^{m,\dagger} = (n+m+1) \mathbf{Y}_{n-1}^{m,\dagger}$ $(m=1,\ldots,n)$, i.e. the hypercomplex differentiation of a basis function delivers a multiple of another basis function one degree lower;\\[-1.0ex]
\item[4.] The polynomials $\mathbf{X}^{n+1,\dagger}_n$ and $\mathbf{Y}^{n+1,\dagger}_n$ are hyperholomorphic constants;\\[-1.0ex]
\item[5.] For $n \geq 1$, we have $\mathcal{P} (\mathbf{X}_n^{l,\dagger}) = \frac{1}{(n+l+2)} \mathbf{X}_{n+1}^{l,\dagger}$ $(l=0,\ldots,n+1)$ and $\mathcal{P} (\mathbf{Y}_n^{m,\dagger}) = \frac{1}{(n+m+2)} \mathbf{Y}_{n+1}^{m,\dagger}$ $(m=1,\ldots,n+1)$, i.e. the hypercomplex primitivation of a basis function delivers a multiple of another basis function one degree upper.
\end{enumerate}

Based on Statement 2, we can easily write down the Fourier expansion of a square integrable ${\mathcal A}$-valued monogenic function. Furthermore, according to the fact that the polynomials $\mathbf{X}_n^{n+1,\dagger}$ and $\mathbf{Y}_n^{n+1,\dagger}$ are hyperholomorphic constants, in \cite{GueJoao2011} we have proved that each $\mathcal{A}$-valued monogenic function can be decomposed in an orthogonal sum of a monogenic "main part" of the function $(\mathbf{g})$ and a hyperholomorphic constant $(\mathbf{h})$. Putting these facts together, next we formulate a modified version of the aforementioned result, which happens to be the more suitable upon further studying the Bloch's theorem. To really understand the above claims, we strongly recomment the reader to consult \cite{GueJoao2011}.
\begin{Lemma} {\rm (Fourier Orthogonal Expansion)} \label{FourierOrthogonalExpansion}
Let $\mathbf{f} \in \mathcal{R}^+(B_r;\mathcal{A})$. The function $\mathbf{f}$ can be represented in the following way
\begin{eqnarray} \label{FourierSeries}
\mathbf{f}(x) &:=& \mathbf{g}(x) \, + \, \mathbf{h}(\underline{x}) \\
&=& \sum_{n=0}^{\infty} \left(\mathbf{X}_n^{0,\dagger,\ast}(x) \, a_n^{0,\ast_r} \, + \, \sum_{m=1}^{n} \left[
\mathbf{X}_n^{m,\dagger,\ast} (x) \, a_n^{m,\ast_r} + \mathbf{Y}_n^{m,\dagger,\ast}(x) \, b_n^{m,\ast_r} \right] \right) \nonumber \\
&+& \sum_{n=0}^{\infty} \left[ \mathbf{X}_n^{n+1,\dagger,\ast}(\underline{x}) \, a_n^{n+1,\ast_r} + \mathbf{Y}_n^{n+1,\dagger,\ast}(\underline{x}) \, b_n^{n+1,\ast_r} \right] \nonumber,
\end{eqnarray}
where for each $n \in \mathbb{N}_0$, $a_n^{0,\ast_r}, a_n^{m,\ast_r}, b_n^{m,\ast_r}$ $(m=1, \dots, n+1)$ are the associated Fourier coefficients.
\end{Lemma}

\begin{Remark}
It is intuitively clear that the method which has led us here implies that the coefficients $a_n^{0,\ast_r}$, $a_n^{m,\ast_r}$ and $b_n^{m,\ast_r}$ $(m=1, \dots, n+1)$ are real constants.
\end{Remark}

In closing this section, let us take a look at the pointwise estimates of the basis polynomials proved in \cite{GueJoao2011}. The proof consists in direct manipulations on certain estimates for the associated Legendre functions (e.g., \cite{Lohofer1998} p.179).
\begin{Proposition} \label{modulusHMP}
For $n \in \mathbb{N}_0$ the polynomials $\mathbf{X}^{l,\dagger}_n$ and $\mathbf{Y}^{m,\dagger}_n$ satisfy the following inequalities:
\begin{eqnarray*}
|\mathbf{X}^{l,\dagger}_n(x)| \; & \leq & \; \frac{1}{2} (n+1) \sqrt{\frac{(n+1+l)!}{(n+1-l)!}} \, |x|^n, \;\;\;\;\;\; l=0,\ldots,n+1 \\
|\mathbf{Y}^{m,\dagger}_n(x)| \; & \leq & \; \frac{1}{2} (n+1) \sqrt{\frac{(n+1+m)!}{(n+1-m)!}} \, |x|^n, \;\;\; m=1,\ldots,n+1.
\end{eqnarray*}
\end{Proposition}

\section{Estimates for $\mathcal{A}$-valued monogenic functions bounded with respect to their hypercomplex derivative}

In the present section we may use the properties of the monogenic polynomials described in the previous section to obtain certain estimates for the Fourier coefficients of an 
$\mathcal{A}$-valued monogenic function by the growth of the maximum modulus of its hypercomplex derivative. Thanks to these properties, such estimates are fairly simple and elegant in comparison with the analogous estimates given by Estermann in the complex case \cite{Estermann1971}.

In the sequel, we shall introduce the notation $\mathcal{M}(\mathbf{f},r) = \max_{B_r} |\mathbf{f}(x)|$, $0 \leq |x| \leq r$ to be used henceforth. This function of $r$ is called the maximum modulus function of $\mathbf{f}$. In preparation for the next section we shall determine an estimate for $\mathcal{M}\left(\mathcal{P}\left\{(\frac{1}{2}\overline{D})\mathbf{f}(x) - (\frac{1}{2}\overline{D})\mathbf{f}(0) \right\}\hspace{-0.1cm},r\right)$ in terms of the $C$-norm of $(\frac{1}{2}\overline{D})\mathbf{f}(x) - (\frac{1}{2}\overline{D})\mathbf{f}(0)$. Next we formulate the result.
\begin{Lemma} \label{Lemma1}
Let $\mathbf{f} \in \mathcal{R}^+(B_r;\mathcal{A})$ such that $\mathbf{f}(0)=\mathbf{0}_{\mathcal{A}}$. Then, for $0 \leq |x| < r$ we have the following inequality:
\begin{eqnarray*}
&& |\mathcal{P}_r \left\{(\frac{1}{2}\overline{D})\mathbf{f}(x) - (\frac{1}{2}\overline{D})\mathbf{f}(0) \right\}| \\
&& \hspace{1.0cm} \leq \frac{2}{\sqrt{3}} \frac{|x|^2 (4|x|^2+9r^2-11|x|r)}{(r-|x|)^3} \,
\mathcal{M}\left((\frac{1}{2}\overline{D})\mathbf{f}(x) - (\frac{1}{2}\overline{D})\mathbf{f}(0),r\right).
\end{eqnarray*}
\end{Lemma}
\begin{proof}
We shall proceed in such a manner that we obtain simultaneously estimates for the Fourier coefficients associated to $\mathbf{f}$, and the existence of Fourier series expansions for the hypercomplex derivative as well as the primitive of an $\mathcal{A}$-valued monogenic function.

To begin with, we consider $\mathbf{f}$ written as in (\ref{FourierSeries}). Since the basis polynomials $\mathbf{X}^{0,\dagger}_n$, $\mathbf{X}^{m,\dagger}_n$ and 
$\mathbf{Y}^{m,\dagger}_n$ are homogeneous (see Statement 1.), a first straightforward computation shows that the value of $\mathbf{f}$ at the
origin is
\begin{eqnarray*}
\mathbf{f}(0) = \frac{1}{2} \sqrt{\frac{3}{\pi \, r^3}} \, (a_0^{0,\ast_r} - a_0^{1,\ast_r} \mathbf{i} - b_0^{1,\ast_r} \mathbf{j}) .
\end{eqnarray*}
By assumption $\mathbf{f}(0)=\mathbf{0}_{\mathcal{A}}$, which yields $a_0^{0,\ast_r} = a_0^{1,\ast_r} = b_0^{1,\ast_r} = 0$. Since the series (\ref{FourierSeries}) is convergent in $L_2(B_r)$, it converges uniformly to $\mathbf{f}$ in each compact subset of $B_r$. Also the series of all partial derivatives converges uniformly to the corresponding partial derivatives of $\mathbf{f}$ in compact subsets of $B_r$. Applying the hypercomplex derivative $\frac{1}{2}\overline{D}$ term by term to the series (\ref{FourierSeries}), and having in mind the underlying orthogonal decomposition it follows formally
\begin{eqnarray*}
&& (\frac{1}{2} \overline{D}) \mathbf{f}(x) \,=\, (\frac{1}{2} \overline{D}) \mathbf{g} + (\frac{1}{2} \overline{D}) \mathbf{h} \\
&& \hspace{-0.5cm} = \; \sum_{n=1}^{\infty} \left[ (\frac{1}{2} \overline{D}) \mathbf{X}_n^{0,\dagger,\ast_r} \, a_n^{0,\ast_r} + \sum_{m=1}^{n} \left( (\frac{1}{2} \overline{D})
\mathbf{X}_n^{m,\dagger,\ast_r} \, a_n^{m,\ast_r} + (\frac{1}{2} \overline{D}) \mathbf{Y}_n^{m,\dagger,\ast_r} \, b_n^{m,\ast_r} \right) \right. \\[0.45ex]
&& \hspace{-0.5cm} + \; \mathbf{0}_{\mathcal{A}}.
\end{eqnarray*}
With these arguments at hand, and using Property 3. we set
\begin{eqnarray} \label{expression1}
&& (\frac{1}{2} \overline{D}) \mathbf{f}(x) - (\frac{1}{2} \overline{D}) \mathbf{f}(0) \nonumber \\
&=& \sum_{n=2}^{\infty} \left[ (\frac{1}{2} \overline{D}) \mathbf{X}_n^{0,\dagger,\ast_r} \, a_n^{0,\ast_r} + \sum_{m=1}^{n} \left( (\frac{1}{2} \overline{D})
\mathbf{X}_n^{m,\dagger,\ast_r} \, a_n^{m,\ast_r} + (\frac{1}{2} \overline{D}) \mathbf{Y}_n^{m,\dagger,\ast_r} \, b_n^{m,\ast_r} \right) \right. \nonumber \\
&=& \sum_{n=2}^{\infty} \left[ \frac{(n+1)}{\|\mathbf{X}_n^{0,\dagger}\|_{L_2(B_r;\mathcal{A};\mathbb{R})}} \mathbf{X}_{n-1}^{0,\dagger} \,
a_n^{0,\ast_r} \right. \nonumber \\
&& \hspace{1.5cm} + \sum_{m=1}^{n} \frac{(n+m+1)}{\|\mathbf{X}_n^{m,\dagger}\|_{L_2(B_r;\mathcal{A};\mathbb{R})}} 
\left( \mathbf{X}_{n-1}^{m,\dagger} \, a_n^{m,\ast_r} + \mathbf{Y}_{n-1}^{m,\dagger} \, b_n^{m,\ast_r} \right).
\end{eqnarray}
By construction, the Fourier coefficients are real-valued, then by applying the linear primitive operator $\mathcal{P}$ term by term to the previous series,  and using Property 5. it follows
\begin{eqnarray} \label{expression2}
&& \mathcal{P}_r \left\{(\frac{1}{2} \overline{D}) \mathbf{f}(x) - (\frac{1}{2} \overline{D}) \mathbf{f}(0) \right\} \nonumber \\
&=& \sum_{n=2}^{\infty} \left[ \frac{(n+1)}{\|\mathbf{X}_n^{0,\dagger}\|_{L_2(B_r;\mathcal{A};\mathbb{R})}} \mathcal{P}(\mathbf{X}_{n-1}^{0,\dagger}) \,
a_n^{0,\ast_r} \right. \nonumber \\
&& \hspace{1.5cm} + \; \sum_{m=1}^{n} \frac{(n+m+1)}{\|\mathbf{X}_n^{m,\dagger}\|_{L_2(B_r;\mathcal{A};\mathbb{R})}} 
\left( \mathcal{P}(\mathbf{X}_{n-1}^{m,\dagger}) \, a_n^{m,\ast_r} + \mathcal{P}(\mathbf{Y}_{n-1}^{m,\dagger}) \, b_n^{m,\ast_r} \right) \nonumber \\
&=& \sum_{n=2}^{\infty} \left[ \frac{1}{\|\mathbf{X}_n^{0,\dagger}\|_{L_2(B_r;\mathcal{A};\mathbb{R})}} \mathbf{X}_n^{0,\dagger} 
\, a_n^{0,\ast_r} \right. + \sum_{m=1}^{n} \frac{\left( \mathbf{X}_n^{m,\dagger} \, a_n^{m,\ast_r} + \mathbf{Y}_n^{m,\dagger} \, b_n^{m,\ast_r} \right)}{\|\mathbf{X}_n^{m,\dagger}\|_{L_2(B_r;\mathcal{A};\mathbb{R})}}.
\end{eqnarray}
We may now prove certain relations between the Fourier coefficients $a_n^{0,\ast_r}$, $a_n^{m,\ast_r}$, and $b_n^{m,\ast_r}$ 
$(m=1,\ldots,n)$ and the factor "$(\frac{1}{2} \overline{D}) \mathbf{f}(x) - (\frac{1}{2} \overline{D}) \mathbf{f}(0)$". Hence multiplying both sides of the expression (\ref{expression1}) by the (orthogonal) homogeneous monogenic polynomials $\{ \mathbf{X}_{n-1}^{0,\dagger}, \mathbf{X}_{n-1}^{m,\dagger}, \mathbf{Y}_{n-1}^{m,\dagger} : m=1,\ldots,n\}$ and integrating over $B_r$, we get the following relations:
\begin{eqnarray*}
a_n^{0,\ast_r} \hspace{-0.325cm} &=& \hspace{-0.325cm} \frac{\|\mathbf{X}_n^{0,\dagger}\|_{L_2(B_r;\mathcal{A};\mathbb{R})}}{\|\mathbf{X}_{n-1}^{0,\dagger}\|^2_{L_2(B_r;\mathcal{A};\mathbb{R})}} \, \frac{1}{(n+1)} \, \int_{B_r} \left\{ (\frac{1}{2} \overline{D}) \mathbf{f}(x) - (\frac{1}{2} \overline{D}) \mathbf{f}(0) \right\} \mathbf{X}_{n-1}^{0,\dagger} \, dV_r \\[0.5ex]
a_n^{m,\ast_r} \hspace{-0.325cm} &=& \hspace{-0.325cm} \frac{\|\mathbf{X}_n^{m,\dagger}\|_{L_2(B_r;\mathcal{A};\mathbb{R})}}{\|\mathbf{X}_{n-1}^{m,\dagger}\|^2_{L_2(B_r;\mathcal{A};\mathbb{R})}} \, \frac{1}{(n+m+1)} \, \int_{B_r} \left\{ (\frac{1}{2} \overline{D}) \mathbf{f}(x) - (\frac{1}{2} \overline{D}) \mathbf{f}(0) \right\} \mathbf{X}_{n-1}^{m,\dagger} \, dV_r \\[0.5ex]
b_n^{m,\ast_r} \hspace{-0.325cm} &=& \hspace{-0.325cm} \frac{\|\mathbf{Y}_n^{m,\dagger}\|_{L_2(B_r;\mathcal{A};\mathbb{R})}}{\|\mathbf{Y}_{n-1}^{m,\dagger}\|^2_{L_2(B_r;\mathcal{A};\mathbb{R})}} \, \frac{1}{(n+m+1)} \, \int_{B_r} \left\{ (\frac{1}{2} \overline{D}) \mathbf{f}(x) - (\frac{1}{2} \overline{D}) \mathbf{f}(0) \right\} \mathbf{Y}_{n-1}^{m,\dagger} \, dV_r,
\end{eqnarray*}
for $m=1,\ldots,n$. We remark that originally the Fourier coefficients are defined as inner products of the function $\mathbf{f}$ and elements of the space $\mathcal{R}^+(\mathbb{R}^3;\mathcal{A};n)$. Now we see that these coefficients, up to a factor, can also be expressed as inner products between the factor "$(\frac{1}{2} \overline{D}) \mathbf{f}(x) - (\frac{1}{2} \overline{D}) \mathbf{f}(0)$" and respectively, $\mathbf{X}_{n-1}^{0,\dagger}, \mathbf{X}_{n-1}^{m,\dagger}$, and $\mathbf{Y}_{n-1}^{m,\dagger}$. Applying Proposition \ref{modulusHMP} the Fourier coefficients satisfy the inequalities
\begin{eqnarray*}
|a_n^{0,\ast_r}| \hspace{-0.325cm} &\leq& \hspace{-0.325cm} \sqrt{\frac{4 \pi r^3}{3}} \frac{\|\mathbf{X}_n^{0,\dagger}\|_{L_2(B_r;\mathcal{A};\mathbb{R})}}{\|\mathbf{X}_{n-1}^{0,\dagger}\|_{L_2(B_r;\mathcal{A};\mathbb{R})}} \, \frac{1}{(n+1)} \, \mathcal{M}\left((\frac{1}{2}\overline{D})\mathbf{f}(x) - (\frac{1}{2}\overline{D})\mathbf{f}(0),r\right) \\[0.5ex]
|a_n^{m,\ast_r}| \hspace{-0.325cm} &\leq& \hspace{-0.325cm} \sqrt{\frac{4 \pi r^3}{3}} \frac{\|\mathbf{X}_n^{m,\dagger}\|_{L_2(B_r;\mathcal{A};\mathbb{R})}}{\|\mathbf{X}_{n-1}^{m,\dagger}\|_{L_2(B_r;\mathcal{A};\mathbb{R})}} \, \frac{1}{(n+m+1)} \, \mathcal{M}\left((\frac{1}{2}\overline{D})\mathbf{f}(x) - (\frac{1}{2}\overline{D})\mathbf{f}(0),r\right) \\[0.5ex]
|b_n^{m,\ast_r}| \hspace{-0.325cm} &\leq& \hspace{-0.325cm} \sqrt{\frac{4 \pi r^3}{3}} \frac{\|\mathbf{Y}_n^{m,\dagger}\|_{L_2(B_r;\mathcal{A};\mathbb{R})}}{\|\mathbf{Y}_{n-1}^{m,\dagger}\|_{L_2(B_r;\mathcal{A};\mathbb{R})}} \, \frac{1}{(n+m+1)} \, \mathcal{M}\left((\frac{1}{2}\overline{D})\mathbf{f}(x) - (\frac{1}{2}\overline{D})\mathbf{f}(0),r\right).
\end{eqnarray*}
One gets, from the estimates of the Fourier coefficients obtained so far, the following estimate
\begin{eqnarray*}
&& |\mathcal{P}_r \left\{(\frac{1}{2} \overline{D}) \mathbf{f}(x) - (\frac{1}{2} \overline{D}) \mathbf{f}(0) \right\}| 
\leq \sum_{n=2}^{\infty} \left[ \frac{1}{\|\mathbf{X}_n^{0,\dagger}\|_{L_2(B_r;\mathcal{A};\mathbb{R})}} 
|\mathbf{X}_n^{0,\dagger}| |a_n^{0,\ast_r}| \right. \\
&+& \left. \sum_{m=1}^{n} \frac{1}{\|\mathbf{X}_n^{m,\dagger}\|_{L_2(B_r;\mathcal{A};\mathbb{R})}} \left( |\mathbf{X}_n^{m,\dagger}| |a_n^{m,\ast_r}| + |\mathbf{Y}_n^{m,\dagger}| |b_n^{m,\ast_r}| \right) \right. \\[0.25ex]
&\leq& \sqrt{\frac{4 \pi r^3}{3}} \, \mathcal{M}\left((\frac{1}{2}\overline{D})\mathbf{f}(x) - (\frac{1}{2}\overline{D})\mathbf{f}(0),r\right) \\
&\times&\sum_{n=2}^{\infty} \left[ \frac{|\mathbf{X}_n^{0,\dagger}|}{\|\mathbf{X}_{n-1}^{0,\dagger}\|_{L_2(B_r;\mathcal{A};\mathbb{R})}} \frac{1}{(n+1)} + \sum_{m=1}^{n} \frac{\left( |\mathbf{X}_n^{m,\dagger}| + |\mathbf{Y}_n^{m,\dagger}| \right)}{\|\mathbf{X}_{n-1}^{m,\dagger}\|_{L_2(B_r;\mathcal{A};\mathbb{R})}} \frac{1}{(n+m+1)} \right..
\end{eqnarray*}
Finally, applying again Proposition \ref{modulusHMP} a straightforward computation shows that
\begin{eqnarray*}
&& \left|\mathcal{P}_r \left\{(\frac{1}{2} \overline{D}) \mathbf{f}(x) - (\frac{1}{2} \overline{D}) \mathbf{f}(0) \right\} \right| \\
&\leq& \sqrt{\frac{2}{3} r^3} \, \mathcal{M}\left((\frac{1}{2} \overline{D}) \mathbf{f}(x) - (\frac{1}{2} \overline{D}) \mathbf{f}(0),r\right) \sum_{n=2}^{\infty} 
\frac{|x|^n}{\sqrt{r^{2n+1}}} \sqrt{\frac{n+1}{n}} (1+2n\sqrt{n+1})\\
&\leq& {\frac{2}{\sqrt{3}}} r \, \mathcal{M}\left((\frac{1}{2} \overline{D}) \mathbf{f}(x) - (\frac{1}{2} \overline{D}) \mathbf{f}(0),r\right) \sum_{n=2}^{\infty} 
\frac{|x|^n}{r^n} (n+1)^2.
\end{eqnarray*}
Using the sum of the series $\sum_{n=2}^{\infty} \left(\frac{|x|}{r}\right)^n (n+1)^2$ for $0 \leq |x| < r$, we arrive at the desired estimate.
\end{proof}

To generalize Bloch's theorem for $\mathcal{A}$-valued monogenic functions defined in the unit ball of the Euclidean space $\mathbb{R}^3$, there's some more groundwork we need to cover. We end this section by proving an estimate for $\mathcal{M}\left((\frac{1}{2}\overline{D})\mathbf{f}(x) - (\frac{1}{2}\overline{D})\mathbf{f}(0),r\right)$ in terms of the $C$-norm of $(\frac{1}{2}\overline{D})\mathbf{f}(x)$.

In the sequel we shall denote by $\mathbf{E}(x)$ the Cauchy kernel function. From (\cite{GS1997}, p.87-88), every $\mathcal{A}$-valued function $\mathbf{f}$ that is monogenic in a neighborhood of the closure $\overline{B_r}$ of $B_r$ may be represented by the Cauchy integral formula as follows
\begin{equation} \label{Cauchyintegralformula}
\mathbf{f}(x) = \frac{1}{4 \pi} \int_{S_r} \mathbf{E}(x-y) \, \mathbf{n}(y) \, \mathbf{f}(y) \, d\sigma_y,
\end{equation}
for $x \in B_r$, where $S_r = \partial B_r$ denotes the boundary of $B_r$, $\mathbf{n}(y)$ the outward pointing normal unit vector at $y \in S_r$ and $d \sigma_y$ the Lebesgue measure on $S_r$.  To show the result, we shall proceed to estimate the difference "$\mathbf{f}(x)-\mathbf{f}(0)$" by using the Cauchy integral formula (\ref{Cauchyintegralformula}). Hence
\begin{eqnarray*}
\mathbf{f}(x) - \mathbf{f}(0) = \frac{1}{4 \pi} \int_{S_r} \left[ \mathbf{E}(x-y)-\mathbf{E}(-y) \right] \mathbf{n}(y) \, \mathbf{f}(y) \, d \sigma_y.
\end{eqnarray*}
We will proceed to use the following inequality from (\cite{GS1989}, p. 50)
\begin{eqnarray*}
|\mathbf{E}(x-y)-\mathbf{E}(-y)| \leq \frac{|x-y| \left( |y|^2 + |y||x-y| + 2 |x-y|^2 \right) |x|}{|x-y|^3 |y|^3}.
\end{eqnarray*}
Indeed, applying this result to
\begin{eqnarray*}
(\frac{1}{2}\overline{D})\mathbf{f}(x) - (\frac{1}{2}\overline{D})\mathbf{f}(0) = \frac{1}{4 \pi} \int_{S_r} \left[ \mathbf{E}(x-y)-\mathbf{E}(-y)\right] \mathbf{n}(y) \, (\frac{1}{2}\overline{D}) \mathbf{f}(y) \, d \sigma_y
\end{eqnarray*}
a direct computation shows that
\begin{eqnarray*}
&& | (\frac{1}{2}\overline{D})\mathbf{f}(x) - (\frac{1}{2}\overline{D})\mathbf{f}(0) | \\
&\leq& \frac{1}{4 \pi} \int_{S_r} \frac{|x-y| \left( |y|^2 + |y||x-y| + 2 |x-y|^2 \right) |x|}{|x-y|^3 |y|^3} \, |\mathbf{n}(y)| \left|(\frac{1}{2}\overline{D})\mathbf{f}(y)\right| \, d \sigma_y \\
&\leq& \frac{1}{4 \pi} \, |x| \, \mathcal{M}\left((\frac{1}{2}\overline{D})\mathbf{f}(x),r\right) \left[ \int_{S_r} 
\left(\frac{1}{r} \frac{1}{|x-y|^2} + \frac{1}{r^2} \frac{1}{|x-y|} + \frac{2}{r^3} \right) \right] d \sigma_y \\
&\leq& \frac{1}{4 \pi} \, |x| \, \mathcal{M}\left((\frac{1}{2}\overline{D})\mathbf{f}(x),r\right) 4 \pi \left(\frac{r}{(r-|x|)^2} + \frac{1}{r-|x|} + \frac{2}{r} \right) \\
&=& \frac{|x|}{(r-|x|)^2} \, \mathcal{M}\left((\frac{1}{2}\overline{D})\mathbf{f}(x),r\right) \left( r + r -|x| + \frac{2}{r} (r-|x|)^2 \right)
\end{eqnarray*}
and, consequently
\begin{eqnarray*}
| (\frac{1}{2}\overline{D})\mathbf{f}(x) - (\frac{1}{2}\overline{D})\mathbf{f}(0) |
&\leq& \frac{|x|}{(r-|x|)^2} \, \mathcal{M}\left((\frac{1}{2}\overline{D})\mathbf{f}(x),r\right) \left( 4r + \frac{2 |x|^2}{r} - 5|x| \right) \\
&\leq& \frac{|x|}{(r-|x|)^2} \, \mathcal{M}\left((\frac{1}{2}\overline{D})\mathbf{f}(x),r\right) \left( 6r - 5|x| \right) \\
&\leq& \frac{6 |x| r}{(r-|x|)^2} \, \mathcal{M}\left((\frac{1}{2}\overline{D})\mathbf{f}(x),r\right).
\end{eqnarray*}
These calculations proved:
\begin{Lemma} \label{Lemma2}
Let $\mathbf{f} \in \mathcal{R}^+(B_r;\mathcal{A})$. Then, for $0 \leq |x| < r$ we have the following inequality:
\begin{eqnarray*}
&& |(\frac{1}{2}\overline{D})\mathbf{f}(x) - (\frac{1}{2}\overline{D})\mathbf{f}(0)| \leq \frac{6 |x| r}{(r-|x|)^2} \,
\mathcal{M}\left((\frac{1}{2}\overline{D})\mathbf{f}(x),r\right).
\end{eqnarray*}
\end{Lemma}

\section{Bloch Theorem}

We state and prove in this section the quaternionic version of Bloch's theorem. It says that if the hypercomplex derivative of an $\mathcal{A}$-valued monogenic function $\mathbf{f}$ at the origin is normalized to magnitude $1$, then there is an open subset of the unit ball that $\mathbf{f}$ maps one-to-one onto some ball of radius at least $R$. We will show that the conclusion of the theorem holds with $R$ equal to $\frac{1}{120} - \frac{31096}{20511149} \sqrt{3} > \frac{1}{150}$. The point here is that $R$ is independent of $\mathbf{f}$. This investigation leads to a number of remarkable results. As a preliminary, we shall prove a simple lemma, which consists of studying the monotony of the real-valued function
\begin{eqnarray*}
g(\rho) = \displaystyle \frac{\rho}{2} - 8 \sqrt{3} \; \displaystyle \frac{\rho^3 r (4 \rho^2 + 9 r^2 - 11 \rho r)}{(r - \rho)^5}.
\end{eqnarray*}
We need to know estimates from below for the maximal value of $g$ in $ (0, r)$ and the location of this maximum. Near to $ x = 0 \; g$ is an increasing function of $\rho$ and if $ \rho$ approaches $ r$ then $ g$ is decreasing. A direct computation shows that
\begin{eqnarray*}
g' \left(\displaystyle \frac{r}{30}\right) > 0, \quad {\rm and} \quad g' \left( \displaystyle \frac{r}{20}\right) < 0.
\end{eqnarray*}
To understand that $ g'$ has only one zero in $(0,r)$ we shall study the second derivative
\begin{eqnarray*}
g''(\rho) = - \frac{48 \sqrt{3} \, \rho r^2 (-7 \rho^2 r + 3 \rho^3 + 5 \rho r^2 + 9 r^3)}{(r - \rho)^7}.
\end{eqnarray*}
We find that $ p(\rho) := 3 \rho^3 - 7 \rho^2 r + 5 \rho r^2 + 9 r^3$ has only one real zero and this zero must be negative. That means $p (\rho) >0 $ in $ (0,r)$ and $g''(\rho) <0$ in $(0,r)$. Therefore, $ \displaystyle \frac{r}{30}$ estimates the zero of $ g'(\rho)$ from below, and $g (\displaystyle \frac{r}{30}) = \left( \displaystyle \frac{1}{60} - \frac{62192}{20511149} \sqrt{3}\right) r$ 
is a lower estimate for the maximum of $g$ in $(0,r)$.\\

Finally, we proved the auxiliary lemma.
\begin{Lemma} \label{PreliminaryLemma1}
Let $g (\rho) = \displaystyle \frac{\rho}{2} - \frac{8 r^3 (\rho^3) (4 \rho^2 + 9 r^2 - 11 \rho r)}{(r - \rho)^5}$. $g$ has only one maximum in $(0,r)$ at $\rho=\rho_{max}$, and it holds 
$g( \rho_{max}) > \left(\displaystyle \frac{1}{60} - \frac{62192}{20511149} \sqrt{3}\right) r$ and $\rho_{max} > \displaystyle \frac{r}{30}$.
\end{Lemma}

Furthermore, with Lemmas \ref{Lemma1} and \ref{Lemma2} in mind we establish the following lemma:
\begin{Lemma} \label{PreliminaryLemma2}
Let $\mathbf{f} \in \mathcal{R}^+(B_r;\mathcal{A})$, which satisfies the normalization condition $\mathcal{M}\left((\frac{1}{2}\overline{D})\mathbf{f}(x),r\right) \leq 2 \left|(\frac{1}{2}\overline{D})\mathbf{f}(a)\right|$ for $a \in B_r$, then the image domain contains balls of radius $R:=\left(\displaystyle \frac{1}{60} - \frac{62192}{20511149} \sqrt{3}\right) r \left|(\frac{1}{2}\overline{D})\mathbf{f}(a)\right|$. (For simplification we remark that $\frac{1}{60} - \frac{62192}{20511149} \sqrt{3} > \frac{1}{75}$).
\end{Lemma}
\begin{proof}
Let $\mathbf{f} \in \mathcal{R}^+(B_r;\mathcal{A})$. We may assume for the moment that $\mathbf{a} = \mathbf{f}(a) = \mathbf{0}_{\mathcal{A}}$. For the sake of simplicity we shall introduce a new notation:
\begin{eqnarray*}
\mathbf{A}(x) := \mathcal{P}_r \left\{ (\frac{1}{2}\overline{D})\mathbf{f}(x) - (\frac{1}{2}\overline{D})\mathbf{f}(0) \right\} = \mathbf{f}(x) - \mathbf{X}^{0,\dagger}_1(x) (\frac{1}{2}\overline{D})\mathbf{f}(0),
\end{eqnarray*}
where $\mathbf{X}^{0,\dagger}_1(x) (\frac{1}{2}\overline{D})\mathbf{f}(0)$ denotes the linear term of the corresponding Taylor series of $\mathbf{f}$. The underlying orthogonality to the non-trivial hyperholomorphic constants highlights in an impressive way a complete analogy to the complex case, where $f-f'(0)z$ is also orthogonal to the constants. Using Lemmas \ref{Lemma1} and \ref{Lemma2}, a first straightforward computation shows that
\begin{eqnarray} \label{EstimateBlochTheorem}
|\mathbf{A}(x)| &=& |\mathcal{P}_r \left\{ (\frac{1}{2}\overline{D})\mathbf{f}(x) - (\frac{1}{2}\overline{D})\mathbf{f}(0) \right\}| \nonumber \\
&\leq& \frac{12}{\sqrt{3}} \frac{|x|^3 \, r (4|x|^2+9r^2-11|x|r)}{(r-|x|)^5} \, \mathcal{M}\left((\frac{1}{2}\overline{D})\mathbf{f}(x),r\right).
\end{eqnarray}
This estimate resembles the one given by Estermann in \cite{Estermann1971}. We use $\mathbf{X}^{0,\dagger}_1(x):=x_0 + \frac{1}{2}x_1 \mathbf{i} + \frac{1}{2} x_2 \mathbf{j}$. Now let $\rho \in (0,r)$, the inequality
\begin{eqnarray*}
|\mathbf{f}(x) - \mathbf{X}^{0,\dagger}_1(x) (\frac{1}{2}\overline{D})\mathbf{f}(0)| \geq \frac{\rho}{2} \Bigl|(\frac{1}{2}\overline{D})\mathbf{f}(0)\Bigr| - |\mathbf{f}(x)|
\end{eqnarray*}
holds for $x$ with $|x|=\rho$. Since by assumption $\mathcal{M}\left((\frac{1}{2}\overline{D})\mathbf{f}(x),r\right) \leq 2 \left|(\frac{1}{2}\overline{D})\mathbf{f}(0)\right|$, it follows from 
(\ref{EstimateBlochTheorem}) that
\begin{eqnarray*}
|\mathbf{f}(x)| \geq \left( \frac{\rho}{2} - 8\sqrt{3} \, \frac{\rho^3 r \, (4 \rho^2 + 9r^2 - 11\rho r)}{(r-\rho)^5} \right) \Bigl|(\frac{1}{2}\overline{D})\mathbf{f}(0)\Bigr|.
\end{eqnarray*}
Hence, extending this argument to an arbitrary point $a \in B_r$ by translation of the corresponding Taylor series of $\mathbf{f}$, and applying Lemma \ref{PreliminaryLemma1} it follows that
\begin{eqnarray*}
|\mathbf{f}(x)-\mathbf{f}(a)| \geq  \left(\displaystyle \frac{1}{60} - \frac{62192}{20511149} \sqrt{3}\right) r\Bigl|(\frac{1}{2}\overline{D})\mathbf{f}(a)\Bigr| =R
\end{eqnarray*}
for all $|x|=r/30$.
\end{proof}

The quaternionic version of Bloch's theorem is contained in the following:
\begin{Theorem}
Let $\mathbf{f} \in \mathcal{R}^+(B_r;\mathcal{A})$. Then its image domain contains balls of radius $R:=\left(\displaystyle \frac{1}{120} - \frac{31096}{20511149} \sqrt{3}\right) 
\mathcal{M}\left(\left|(\frac{1}{2}\overline{D})\mathbf{f}(x)\right|(1-|x|),r\right) > \frac{1}{150} \left|(\frac{1}{2}\overline{D})\mathbf{f}(0)\right|$.
\end{Theorem}
\begin{proof}
By the same reasoning as in \cite{Estermann1971}, to every function $\mathbf{f} \in \mathcal{R}^+(B_r;\mathcal{A})$ we assign the function $|(\frac{1}{2}\overline{D})\mathbf{f}(x)|(1-|x|)$, which is continuous on $\overline{B_r}$. It assumes its maximum at a point $q \in B_r$. With $t:=\frac{1}{2}(1-|q|)$, we have
\begin{eqnarray*}
\mathcal{M}\left(\Bigl|(\frac{1}{2}\overline{D})\mathbf{f}(x)\Bigr|(1-|x|),r\right) = 2t \Bigl|(\frac{1}{2}\overline{D})\mathbf{f}(q)\Bigr|, \quad B_t(q) \subset B_r
\end{eqnarray*}
and, $1-|x| \geq t$ for $x \in B_t(q)$. In the first place, we note, that from the relation $|(\frac{1}{2}\overline{D})\mathbf{f}(x)|(1-|x|) \leq 2 t |(\frac{1}{2}\overline{D})\mathbf{f}(q)|$, it follows 
that $|(\frac{1}{2}\overline{D})\mathbf{f}(x)| \leq 2 |(\frac{1}{2}\overline{D})\mathbf{f}(q)|$ for all $x \in B_t(q)$. Hence, from Lemma \ref{PreliminaryLemma2}, the image domain of $\mathbf{f}$ contains balls of radius $R:=\left(\displaystyle \frac{1}{60} - \frac{62192}{20511149} \sqrt{3}\right) t \left|(\frac{1}{2}\overline{D})\mathbf{f}(q)\right|$.
\end{proof}

By combining all previous results, the quaternionic version of Bloch's theorem reads as follows:
\begin{Theorem} {\rm (Bloch Theorem)}
Let $\mathbf{f} \in \mathcal{R}^+(B_r;\mathcal{A})$ such that $|(\frac{1}{2} \overline{D})\mathbf{f}(0)| = 1$. Then its image domain contains balls of radius $\displaystyle \frac{1}{120} - \frac{31096}{20511149} \sqrt{3} > \frac{1}{150}$.
\end{Theorem}

\medskip

Ultimately, it has to be studied in the future how the value of the radius for which the conclusion of the theorem holds can be improved.

\end{document}